\documentclass[reqno, 12pt]{amsart}
\usepackage[letterpaper,hmargin=1in,vmargin=1.1in]{geometry}
\usepackage{amsmath, amssymb, amsthm, verbatim, bbm, url} 

\newcommand \N {\mathbb{N}}
\newcommand \R {\mathbb{R}}
\newcommand \C {\mathbb{C}}
\newcommand \Z {\mathbb{Z}}

\newcommand \Oh {\mathcal{O}}

\newcommand \D {\partial}
\newcommand \eps {\varepsilon}

\DeclareMathOperator \im {Im}

\DeclareMathOperator \sgn {sgn}
\DeclareMathOperator \Res {Res}

\DeclareMathOperator \Tr {Tr}

\DeclareMathOperator \card{card}
\DeclareMathOperator \vol {vol}

\newtheorem{lem}{Lemma}
\newtheorem*{thm}{Theorem}

\theoremstyle{definition}

\numberwithin{equation}{section}
\numberwithin{lem}{section}
\numberwithin{Defn}{section}

\title
{Resonant uniqueness of radial semiclassical Schr\"odinger operators}

\author[Kiril Datchev]
{Kiril Datchev}
\author[Hamid Hezari]
{Hamid Hezari}
\address{Mathematics Department, Massachusetts Institute of Technology, Cambridge, MA
02139.}
\email{datchev@math.mit.edu}
\email{hezari@math.mit.edu}
\thanks{The first author is partially supported by a National Science Foundation postdoctoral fellowship, and the second author is partially supported by the National Science Foundation under
grant DMS-0969745. The authors are grateful for the hospitality of the Mathematical Sciences Research Institute.}
\date{October 17, 2011}

\begin{document}

\begin{abstract}
We prove that radial, monotonic, superexponentially decaying potentials in $C^\infty(\R^n)$, $n \ge 1$ odd, are determined by the resonances of the associated semiclassical Schr\"odinger operator among all superexponentially decaying potentials in $C^\infty(\R^n)$.
\end{abstract}

\maketitle

\section{Introduction}

Given $V\in C^\infty(\R^n;\R)$, the semiclassical inverse spectral problem asks: what information about $V$ can be recovered from the asymptotics of the spectrum of $-h^2\Delta + V$ as $h \to 0$? In the case when the spectrum is discrete, various positive results have been given since the papers of Sj\"ostrand \cite{sjo} and Iantchenko-Sj\"ostrand-Zworski \cite{ISZ}, including many sufficient conditions under which the potential can be determined, beginning with the work of Guillemin-Uribe \cite{gu07} (see also \cite{h09,cg,c,gw10,dhv}).

In this paper we consider potentials $V \in C^\infty(\R^n; [0,\infty))$ satisfying
\begin{equation}\label{e:ass}
|V(x)| \le A \exp(-B |x|^{1+\eps}), \qquad |\D^\alpha V(x)| \le C_\alpha,
\end{equation}
for every multiindex $\alpha$ and for some $\eps,A,B,C_\alpha > 0$.
For such potentials the spectrum of $-h^2\Delta + V$ is continuous and equals $[0,\infty)$, and hence contains no (further) information about $V$.  In this setting \textit{resonances} replace the discrete data of eigenvalues. For $n$ odd they are defined as the poles of the meromorphic continuation of $R_V(\lambda) = (-h^2\Delta + V - \lambda^2)^{-1} \colon L^2_{\textrm{comp}} \to L^2_{\textrm{loc}}$ from $\{\im \lambda > 0\}$ to $\C$.

We prove that potentials which are radial and monotonic are determined by their resonances among all potentials satisfying \eqref{e:ass}.

\begin{thm}
Let $n \ge 1$ be odd, and let $V_0,V \in C^\infty(\R^n;[0,\infty))$ satisfy \eqref{e:ass}. Suppose further that $V_0(x) = R(|x|)$, and $R'(r)$ vanishes only at $r=0$ and whenever $R(r) = 0$. Suppose that the resonances of $-h^2\Delta + V_0(x)$ agree with the resonances of $-h^2 \Delta + V(x)$, up to $o(h^2)$ uniformly on compact sets, for $h \in \{h_j\}_{j=1}^\infty$ for some sequence $h_j \to 0$. Then there exists $x_0 \in \R^n$ such that $V(x) = V_0(x-x_0)$.
\end{thm}

The strong symmetry assumption on $V_0$ is crucial to our argument. To our knowledge, such symmetry assumptions are always needed for uniqueness results for inverse resonance problems (and also for inverse spectral problems, see the introduction of \cite{dhv} for a discussion). The strongest previous results are those for the nonsemiclassical Schr\"odinger problem when $n=1$. In \cite{Zworski:isopolar} Zworski proves that a compactly supported even potential $V$ is determined from the resonances of $-\frac {d^2}{dx^2} + V$, and in \cite{K} Korotyaev shows that a potential which is not necessarily even is determined by some additional scattering data. In the present paper we study the semiclassical problem for general odd dimensions, for not necessarily compactly supported potentials, and assume a priori only that $V_0$ and not necessarily $V$ is radial, and we use only resonances to determine $V$.

Analogous results hold in the case of obstacle scattering. Hassell and Zworski \cite{hz99} show that a ball is determined by its Dirichlet resonances among all compact obstacles in $\R^3$. Christiansen \cite{chr} extends this result to multiple balls, to higher odd dimensions, and to Neumann resonances. As in the present paper, the proofs use two trace invariants and isoperimetric-type inequalities, although the invariants and inequalities are different here. The case of an analytic obstacle with two mutually symmetric connected components is treated by Zelditch \cite{z04} by using the singularities of the wave trace generated by the bouncing ball between the two components. There is also a large literature of inverse scattering results where data other than the resonances are used: see for example \cite{melbook}.

Our proof is based on recovering and analyzing the first two integral invariants of the Helffer-Robert semiclassical trace formula (\cite[Theorem 1.1]{robert}, see also \cite[Proposition 5.3]{hr} and \cite[\S 10.5]{gs}):
\begin{align}
\Tr(f(-h^2 \Delta &+ V) - f(-h^2\Delta)) = \label{e:hsres}\\\nonumber\frac 1 {(2\pi h)^n} &\left(\int_{\R^{2n}} f(|\xi|^2 + V) - f(|\xi|^2)dxd\xi + \frac{h^2}{12} \int_{\R^{2n}} |\nabla V|^2 f^{(3)}(|\xi|^2 + V)dxd\xi + \Oh(h^4)\right),
\end{align}
valid for $f \in \mathcal{S}(\R)$. This is analogous to the approach taken by Colin de Verdi\`ere \cite{c}, by Guillemin-Wang \cite{gw10}, and by Ventura and the authors \cite{dhv} for the problem of recovering the potential from the discrete spectrum. In \S\ref{s:trace} we sketch a proof of \eqref{e:hsres} for the reader's convenience.

To express the left hand side of \eqref{e:hsres} in terms of the resonances of $-h^2\Delta + V$, we use Melrose's Poisson formula (\cite{Melrose:Trace}), an extension of the formula of Bardos-Guillot-Ralston (\cite{bgr}), adapted to $V$ satisfying \eqref{e:ass} by S\'a Barreto-Zworski (\cite{SabZwo, SabZwo2}):
\begin{equation}\label{e:melrose}
2\Tr\left(\cos(t\sqrt{-h^2 \Delta + V}) - \cos(t\sqrt{-h^2 \Delta})\right) = \sum_{\lambda \in \Res} e^{-i|t|\lambda}, \qquad t \ne 0,
\end{equation}
where $\Res$ denotes the set of resonances of $-h^2\Delta + V$, included according to multiplicity, and the equality is in the sense of distributions on $\R \setminus 0$. 
In fact, it is because \eqref{e:melrose} is known to hold only under the decay assumption \eqref{e:ass} that we make that assumption.
When $n=1$ a stronger trace formula, valid for all $t \in \R$, is known: see for example \cite[page 3]{Zworski:XEDP}. When $n$ is even, the meromorphic continuation of the resolvent is not to $\C$ but to the Riemann surface of the logarithm, and as a result Poisson formul\ae ~ for resonances are more complicated and contain error terms which we have not been able to treat.

If one replaces the assumption that the resonances of the two potentials agree with an assumption that the scattering phases agree, it should be possible to weaken the decay condition \eqref{e:ass} (allowing longer-range potentials) and to consider even dimensions as well as odd dimensions. In this case one would use the Birman-Krein formula (see e.g. \cite[\S 4.1]{melbook} and \cite[Chapter 8]{yafaev}) in place of \eqref{e:melrose} to express the trace of the wave group in terms of the scattering phase. The key point is that the Birman-Krein formula holds more generally than \eqref{e:melrose}, and the proof would otherwise be similar to the proof given below, but the assumption that the two potentials have the same scattering phase is stronger than the assumption that they have the same resonances.

We are grateful to Maciej Zworski for his interest in this project and for helpful discussions regarding Poisson formul\ae~for resonances. Thanks also to the anonymous referee for suggesting that we add an expanded discussion of \eqref{e:hsres}.

\section{Proof of the theorem}\label{s:intinv}

From \eqref{e:melrose}, it follows that if 
\begin{equation}\label{e:moments}\hat g \in C_0^\infty(\R \setminus 0)  \textrm{ is even,}\end{equation}
 then
\begin{equation}\label{e:melrose2}
\Tr(g(\sqrt{-h^2 \Delta + V}) - g(\sqrt{-h^2\Delta})) = \frac 1 {4\pi}\sum_{\lambda \in \Res} \int_\R e^{-i|t|\lambda}\hat g(t)dt.
\end{equation}

 Now setting the right hand sides of \eqref{e:melrose2} and \eqref{e:hsres} equal and taking $h \to 0$, we find that
\begin{equation}\label{e:traceinv}
\int_{\R^{2n}}  f(|\xi|^2 + V) - f(|\xi|^2)dxd\xi, \qquad  \int_{\R^{2n}} |\nabla V|^2 f^{(3)} (|\xi|^2 + V)dxd\xi
\end{equation}
are resonant invariants (i.e. are determined by knowledge of the resonances up to $o(h^2)$) provided that $f(\tau^2) = g(\tau)$ for all $\tau$ and for some $g$ as in \eqref{e:moments}. Taylor expanding, we write the first invariant as
\[
\sum_{k=1}^m \frac 1 {k!} \int _{\R^n} f^{(k)}(|\xi|^2)d\xi \int_{\R^n} V(x)^k dx + \int_{\R^{2n}}\frac{V(x)^{m+1}} {m!} \int_0^1(1-t)^mf^{(m+1)}(|\xi|^2 + tV(x))dtdxd\xi.
\]
Replacing $f$ by $f_\lambda$, where $f_\lambda(\tau) = f(\tau/\lambda)$ (note that  $g_\lambda(\tau) = f_\lambda(\tau^2)$ satisfies \eqref{e:moments}) gives
\[
\sum_{k=1}^m \lambda^{n/2-k}\frac 1 {k!} \int _{\R^n} f^{(k)}(|\xi|^2)d\xi \int_{\R^n} V(x)^k dx + \Oh(\lambda^{n/2-m-1})
\]
Taking $\lambda \to \infty$ and $m \to \infty$ we obtain the invariants
\[
\int _{\R^n} f^{(k)}(|\xi|^2)d\xi \int_{\R^n} V(x)^k dx,
\]
for every $k \ge 1$.
\begin{lem}
There exists $g$ satisfying \eqref{e:moments} such that if $f(\tau^2) = g(\tau)$, then
\[
\int _{\R^n} f^{(k)}(|\xi|^2)d\xi \ne 0,
\]
provided $k \ge n$.
\end{lem}
\begin{proof}
Passing to polar coordinates, and writing $f^{(k)}(\tau^2) = \sum_{j=1}^k c_j  g^{(j)}(\tau) \tau^{j-2k}$, we obtain
\[
\int _{\R^n} f^{(k)}(|\xi|^2)d\xi = \lim_{\eps \to 0^+} \sum_{j=1}^k c_j \int _\eps^\infty g^{(j)}(\tau)\tau^{j-2k+n-1}d\tau.
\]
We next integrate each integral by parts $2k-j-n$ times to obtain
\begin{equation}\label{e:evenid}
\int _{\R^n} f^{(k)}(|\xi|^2)d\xi  =  A \int_0^\infty g^{(2k-n)}(\tau) \tau^{-1}d\tau + B g^{(2k-n)}(0) = A \int_0^\infty g^{(2k-n)}(\tau) \tau^{-1}d\tau
\end{equation}
for some constants $A,B$. Note that all negative powers of $\eps$ in the boundary terms must cancel when summed in $j$, since the left hand side is a finite integral, and that $g^{(2k-n)}(0) = 0$ since $2k-n$ is odd. To prove that $A \ne 0$, we observe that the identity \eqref{e:evenid} holds for $f(\tau) = e^{-\tau}$, $g(\tau) = e^{-\tau^2}$, and that in that case $\int f^{(k)}(|\xi|^2)d\xi = (-1)^k \pi^{n/2}$.
Now,
\[2 \int_0^\infty g^{(2k-n)}(\tau) \tau^{-1}d\tau = \int_{-\infty}^\infty g^{(2k-n)}(\tau) \tau^{-1}d\tau = \frac { i^{2k-n+1}}2 \int_{-\infty}^\infty t^{2k-n}\hat g(t) \sgn t dt,\]
where we used the oddness of $g^{(2k-n)}$ followed by Plancherel's theorem. To make the final expression nonzero it suffices to take $g$ such that $\hat g$ is nonnegative and not identically $0$.
\end{proof}

This shows that 
\begin{equation}
\label{e:vk}\int_{\R^{n}} V(x)^k dx = \int_{\R^{n}} V_0(x)^k dx
\end{equation}
for every $k \ge n$, and a similar analysis of the second invariant of \eqref{e:traceinv} proves that
\begin{equation}\label{e:vk2}
\int_{\R^{n}} V(x)^k |\nabla V(x)|^2 dx = \int_{\R^{n}} V_0(x)^k |\nabla V_0(x)|^2 dx
\end{equation}
for every $k \ge n$.

We rewrite the invariant \eqref{e:vk} using $V_*dx$, the pushforward of Lebesgue measure by $V$, as
\begin{equation}\label{e:deltasum}
\int_{\R^{n}} V(x)^k dx = \int_\R s^k(V_*dx)_s = i^{k} \widehat{V_*dx}^{(k)}(0).
\end{equation}
Since $V$ and $V_0$ are both bounded functions, the pushforward measures are compactly supported and hence have entire Fourier transforms, and we conclude that
\[
V_*dx = {V_0}_*dx + \sum_{k=0}^{n-1} c_k \delta^{(k)}_0= {V_0}_*dx +  c_0 \delta_0.
\]
For the first equality we used the invariants \eqref{e:deltasum}, and for the second the fact that $V_*dx$ is a measure.
In other words
\begin{equation}\label{e:vols}
\vol(\{V > \lambda\}) = \vol(\{V_0 > \lambda\})
\end{equation}
whenever $\lambda > 0$. Moreover, this shows that $V_*dx$ is absolutely continuous on $(0,\infty)$, and so by Sard's lemma the critical set of $V$ is Lebesgue-null on $V^{-1}((0,\infty))$. As a result we may  use the coarea formula\footnote{If $n=1$ we put $\int_{\{V=s\}} |\nabla V|^{-1} dS =\sum_{x \in V^{-1}(s)} |V'(x)|^{-1}.$}  to write
\[
V_*dx = \int_{\{V=s\}} |\nabla V|^{-1} dS ds, \qquad \textrm{on }(0,\infty)
\]
and to conclude that 
\begin{equation}\label{e:inv1}
 \int_{\{V=s\}} |\nabla V|^{-1} dS = \int_{\{V_0=s\}} |\nabla V_0|^{-1} dS
 \end{equation}
for almost every $s>0$.
Similarly, rewriting the invariants \eqref{e:vk2} as
\[
\int_{\R^{n}} V(x)^k |\nabla V(x)|^2 dx = \int_\R s^k \int_{\{V=s\}} |\nabla V|dSds, 
\]
we find that
\begin{equation}\label{e:inv2}
 \int_{\{V=s\}} |\nabla V|dS =  \int_{\{V_0=s\}} |\nabla V_0|dS, \qquad s>0.
\end{equation}

From the Cauchy-Schwarz inequality, \eqref{e:inv1} and \eqref{e:inv2} we find that
\begin{equation}\label{e:cs}
\left(\int_{\{V=s\}}1 dS\right)^2 \le \int_{\{V=s\}} |\nabla V|^{-1}dS\int_{\{V=s\}} |\nabla V|dS  = \left(\int_{\{V_0=s\}} 1dS\right)^2,
\end{equation}
for almost every $s>0$,
where for the last equality we used the fact that $\nabla V_0 = R'$ is constant on level sets of $V_0$. By assumption these level sets $\{V_0=s\}$ are spheres, and by \eqref{e:vols} the volumes of their interiors equal those of $\{V=s\}$.
Hence by the isoperimetric inequality\footnote{If $n=1$ the `isoperimetric inequality' states that if an open set $U$ has the same measure as an interval $I$, then $\card \D U \ge \card \D I$, with equality if and only if $U$ is an interval.} the level sets $\{V=s\}$ for almost every $s$ are spheres also, and furthermore equality is attained in \eqref{e:cs}. Consequently, from the Cauchy-Schwarz equality we conclude that $|\nabla V|$ and $|\nabla V|^{-1}$ are proportional on these level sets $\{V=s\}$, with
\begin{equation}\label{e:firstorderpde}|\nabla V(x)|^2 = R'(R^{-1}(V(x))).\end{equation}
Because by assumption the right hand side does not vanish for $x$ such that $V(x) \in (0,\max V_0)$, we may conclude that the same is true of the left hand side and that \eqref{e:firstorderpde} holds for all $x \in V^{-1}((0,\max V_0)).$ Solving \eqref{e:firstorderpde} along gradient flowlines as in \cite[\S3]{dhv} gives the conclusion.

\section{The semiclassical trace formula}\label{s:trace}
In this section we  sketch a proof of \eqref{e:hsres} for the reader's convenience.
Let 
\[
P_0 = -h^2\Delta + 1 + |\min V|, \qquad P = P_0 + V,
\]
where 
$|\D^\alpha V(x)| \le C_\alpha(1 + |x|^2)^{-\eps - n/2}$ for any multiindex $\alpha$, for some $C_\alpha,\eps > 0$. We prove that if $f$ is Schwartz on $\R$ then $f(P) - f(P_0)$ is trace class and \eqref{e:hsres} holds.

By \cite[Theorem 8.7]{ds}\footnote{In \cite{ds} the result is stated for $f \in C_0^\infty(\R)$. The same proof gives the result for $f$ Schwartz if one defines an almost analytic extension of $f$ which is Schwartz on $\C$, for example as in \cite[(3.1.22)]{ez}. See also \cite[Theorem 13.9]{ez} for a similar result.} 
it follows that $f(P) - f(P_0)$  is a semiclassical pseudodifferential operator with order function $(1 + |\xi|^2)^{-M}$ for any $M \in \N$, that is to say it is the semiclassical quantization of a symbol function $\sigma(f(P) - f(P_0))$ satisfying 
\[
|\D_x^\alpha \D_\xi^\beta \sigma(f(P) - f(P_0))(x,\xi)| \le C_{\alpha,\beta,M} (1 + |\xi|^2)^{-M},
\]
for any multiindices $\alpha,\beta$ and any $M \in \N$ (see e.g. \cite[Chapter 7]{ds} and \cite[Chapter 4]{ez} for background on pseudodifferential operators).

Since the integral kernel of $f(P) - f(P_0)$ is then continuous, once we show that  $f(P) - f(P_0)$ is trace class it will follow that the trace is given by the integral of this kernel along the diagonal (see e.g \cite[Supplement 1, Theorem 4.1]{bs}), namely that
\[
\Tr (f(P) - f(P_0)) = \frac 1 {(2 \pi h)^n} \int_{\R^{2n}}  \sigma(f(P) - f(P_0))(x,\xi) dxd\xi.
\]
The formula \eqref{e:hsres} then follows from a computation of the asymptotics of $\sigma(f(P) - f(P_0))(x,\xi)$ as $h \to 0$. This can be found in \cite[Proposition 5.3]{hr}, in \cite[Chapter 8]{ds}, and in \cite[\S10.5]{gs} (the formulas for the coefficients in \eqref{e:hsres} come from \cite{gs}).

It remains to show that $f(P) - f(P_0)$ is trace class.
Using Fourier inversion and Duhamel's formula we write
\[
f(P) - f(P_0) = \frac 1 {2\pi} \int \hat f(t) (e^{itP} - e^{itP_0})dt=  \frac i {2\pi} \int \hat f(t) e^{itP_0}\int_0^t e^{-isP_0}Ve^{isP} dsdt.
\]
Using $e^{itP_0} = -i P_0^{-1} \D_t e^{itP_0}$ and integrating by parts gives
\begin{align}
f(P) - f(P_0) &= -\frac 1 {2\pi} \int \hat f(t) P_0^{-1} \left(\D_t e^{itP_0}\right)\int_0^t e^{-isP_0}Ve^{isP}dsdt \nonumber\\
\label{e:tr1}&= \frac 1 {2\pi} \int  \hat f(t) P_0^{-1} Ve^{itP}dt + \frac 1 {2\pi}\int \hat f'(t)   e^{itP_0} \int_0^t e^{-isP_0}P_0^{-1}Ve^{isP}ds dt.
\end{align}
The first term of \eqref{e:tr1} can be written
\[
 \frac 1 {2\pi} \int  \hat f(t) P_0^{-1} Ve^{itP}dt  = P_0^{-1}  V f(P).
\]
But $P_0^{-1} V f(P)$ is a pseudodifferential operator with order function $(1 + |x|^2)^{-\eps - n/2}(1 + |\xi|^2)^{-M}$ for any $M \in \N$ because the order function of a composition is the product of the order functions (see e.g. \cite[Chapter 7]{ds} or \cite[Theorem 4.18]{ez}), and since $(1 + |x|^2)^{-\eps - n/2}(1 + |\xi|^2)^{-M} \in L^1(\R^{2n})$ for $M$ large enough it follows that $P_0^{-1} V f(P)$ is trace class by \cite[Theorem 9.4]{ds}.

Since $P_0^{-1} V$ has order function $(1 + |x|^2)^{-\eps - n/2}(1 + |\xi|^2)^{-1}$, it is trace class for $n=1$, and since Schr\"odinger propagators are unitary on $L^2(\R^n)$, we conclude that the second term of \eqref{e:tr1} is also trace class for $n=1$, so that $f(P) - f(P_0)$ is trace class.

If $n \ge 2$, the same computation allows us to write the second term of \eqref{e:tr1} as
\begin{align}
\frac 1{2\pi}\int  \hat f'(t)  & e^{itP_0} \int_0^t e^{-isP_0}P_0^{-1}Ve^{isP}ds dt = \nonumber\\
\label{e:tr2}&- \frac i {2\pi} \int\hat f'(t) P_0^{-2} Ve^{itP}dt - \frac i {2\pi} \int\hat f''(t)   e^{itP_0} \int_0^t e^{-isP_0}P_0^{-2}Ve^{isP}dsdt.
\end{align}
The first term of \eqref{e:tr2} can be analyzed in the same manner as the first term of \eqref{e:tr1}, and is trace class in any dimension. The second term of \eqref{e:tr2} is trace class when $n \le 3$, and the case $n \ge 4$ can be treated by iterating this procedure $\lfloor (n+1)/2\rfloor$ times.

\end{document}